\documentclass[12pt]{article}
\usepackage[english]{babel}
\usepackage[latin1]{inputenc}
\usepackage{amsfonts}
\usepackage{amssymb}
\usepackage{amsmath}
\usepackage{amsthm}

\newcommand{\R}{\mathbb{R}}
\newcommand{\C}{\mathbb{C}}
\newtheorem{theorem}{Theorem}[section]

\newtheorem{notation}[theorem]{Notation}
\newtheorem{cor}[theorem]{Corollary}
\newtheorem{rem}[theorem]{Remark}
\newtheorem{defi}[theorem]{Definition}
\newtheorem{exa}[theorem]{Example}
\newcommand\bfootnote[1]{%
  \begingroup
  \renewcommand\thefootnote{}\footnote{#1}%
  \addtocounter{footnote}{-1}%
  \endgroup
}
\title{Minimal Projections with respect to Numerical Radius}
\author{Asuman G\"{u}ven  Aksoy and Grzegorz Lewicki}
\date{}
\begin{document}
\maketitle
\nocite{*}
\begin{abstract}
In this paper we survey some results on minimality of projections with respect to numerical radius.  We note that in the cases $L^p$, $p=1,2,\infty$, there is no difference between the minimality of projections measured either with respect to operator norm or with respect to numerical radius.  However, we give an example of a projection from $l^p_3$ onto a two-dimensional subspace which is minimal with respect to norm, but not with respect to numerical radius for $p\neq 1,2,\infty$.  Furthermore, utilizing a theorem of Rudin and motivated by Fourier projections, we give a criterion for minimal projections, measured in numerical radius. Additionally, some results concerning strong unicity of minimal projections with respect to numerical radius are given.
\bfootnote{2010 {\em Mathematics Subject Classification}: Primary 41A35,41A65, Secondary 47A12.}
\end{abstract}

\section{Introduction}
A projection from a normed linear space $X$ onto a subspace $V$ is a bounded linear operator $P:X \to V$ having the property that $P_{|_{V}}=I$.  $P$ is called a \emph{minimal projection} if $\|P\|$ is the least possible.  Finding a minimal projection of  the least norm has its obvious connection to approximation theory, since for any $P \in \mathcal{P}(X,V)$, the set of all projections from $X$ onto $V$, and $x \in X$, from the inequality:
\begin{equation}
\|x-Px\|\leq (\|Id -P\|)\, dist(x,V) \leq (1+ \|P\|)\, dist(x,V),
\end{equation}
one can deduce that $Px$ is a good approximation to $x$ if $\|P\|$ is small. 
Furthermore, any minimal projection $P$ is an extension of $Id_V$ to the space $X$  of the smallest possible norm, which can be interpreted as a Hahn-Banach extensions.  In general, a given subspace will not be the range of a projection of norm 1, and the projection of least norm is difficult to discover even if its existence is known \textit{a priori}.  For example, the minimal projection of $C[0,1]$ onto the subspace $\Pi_3$ of polynomials of degree $\leq 3$ is unknown.  For an explicit determination of the projection of minimal norm from the subspace $C[-1,1]$ onto $\Pi_2$, see \cite{Chalmers-Metcalf}.
However, it is known that, see \cite{Cheney-Morris}, for a Banach space $X$ and subspace $V \subset X$, $V=Z^*$ for some Banach space $Z$, then there exists a minimal projection $P:X \to Z$.  A well known example of a minimal projection, \cite{Lozinskii}, 
is Fourier projection $F_m:C(2\pi)\to \Pi_M:=\text{span}\{1,\sin x,\cos x, \dots, \sin mx,\cos mx\}$ defined as 
\begin{equation}
F_m(f)=\sum_{k=0}^m \alpha_k \cos kx+\sum_{k=0}^m \beta_k \sin kx
\end{equation}
where $\alpha_k,\beta_k$ are Fourier coefficients and $C(2\pi)$ denotes $2\pi-$periodic, real-valued functions equipped with the sup norm.  For uniqueness of minimality of Fourier projection also see \cite{Fisher-Morris-Wulbert}.  Let $X$ be a Banach space over $\R$ or $\C$.  We write $B_{X}(r)$ for a closed ball with radius $r>0$ and center at $0$ ($B_X$ if $r=1$) and $S_X$ for the unit sphere of $X$.  The dual space of $X$ is denoted by $X^*$ and the Banach algebra of all continuous linear operators going from $X$ into a Banach space $Y$ is denoted by $B(X,Y)$ ($B(X)$ if $X=Y$).
\par
The numerical range of a bounded linear operator $T$ on $X$ is a subset of a scalar field, constructed in such a way that it is related to both algebraic and norm structures of the operator, more precisely:
\begin{defi}
The {\emph numerical range} $T \in \vec{B}(X)$ is defined by 
\begin{equation}
W(T)=\{x^*(Tx):x \in S_X,x^* \in S_{X^*},x^*(x)=1\}.
\end{equation}
\end{defi}
Notice that the condition $x^*(x)=1$ gives us that $x^*$ is a norm attaining functional.
\par
The concept of a numerical range comes from Toepliz's original definition of the \emph{field of values} associated with a matrix, which is the image of the unit sphere under the quadratic form induced by the matrix $A$:
\begin{equation}
F(A)=\{x^*Ax:\|x\|=1,x \in \C^n\},
\end{equation}
where $x^*$ is the original conjugate transform and $\|x\|$ is the usual Euclidean norm.  It is known that the classical numerical range 
of a matrix always contains the spectrum,  and as a result study of numerical range can help understand properties that depend on the location of the eigenvalues such as stability and non-singularity of matrices.  In case $A$ is a normal matrix, then the numerical range is the polygon in the complex plane whose vertices are eigenvalues of $A$.  In particular, if $A$ is hermitian, then the polygon reduces to the segment on the real axis bounded by the smallest and largest eigenvalue, which perhaps explains the name numerical range.
\par
The \emph{numerical radius} of $T$ is given by 
\begin{equation}
\|T\|_w=\sup\{|\lambda|:\lambda \in W(T)\} .
\end{equation}
Clearly $\|T||_w$ is a semi-norm on $B(X)$ and $\|T\|_w \leq \|T||$ for all $T \in B(X)$.  For example, if we consider $T: \C^n \to \C^n$ as  a right shift operator 
$$
T(f_1,f_2\dots,f_n)=(0,f_1,f_2,\dots,f_{n-1})
$$ 
then $\langle Tf,f\rangle=f_1\overline{f}_2+f_2\overline{f}_3+\dots f_{n-1}\overline{f}_n$ and consequently to find $\|T\|_w$ we must find $\sup\{|f_1||f_2|+\dots+|f_{n-1}||f_n|\}$ subject to the condition $\displaystyle \sum_{i=1}^n |f_i|^2=1$.  The solution to this ``Lagrange multiplier" problem is 
\begin{equation}
 \|T\|_w=\cos\left(\frac{\pi}{n+1}\right).
\end{equation}
\par
The \emph{numerical index} of $X$ is then given by 
\begin{equation}
n(X)=\inf\left\{\|T\|_w:T\in S_{B(X)}\right\}.
\end{equation}
Equivalently, the numerical index $n(X)$ is the greatest constant $k\geq 0$ such that $k\|T\|\leq \|T\|_w$ for every $T \in B(X)$.  Note also that $0\leq n(X)\leq 1$ and $n(X)>0$ if and only if $\|\cdot \|_w$ and $\|\cdot \|$ are equivalent norms.  The concept of numerical index was first introduced by Lumer \cite{lumer} in 1968.  Since then much attention has been paid to the constant of equivalence between the numerical radius and the usual norm of the Banach algebra of all bounded linear operators of  a Banach space. 
Two classical books devoted to these concepts are \cite{Bonsall-Duncan1} and \cite{Bonsall-Duncan2}. For more recent results we refer the reader to \cite{Aksoy-Lewicki3},\cite{Martin},\cite{Martin-Meri-Popov} and \cite{Gustafson-Rao}.
\par
In this paper, we study minimality of projections with respect to numerical radius.  Since operator norm of $T$ is defined as $\|T\|=\sup|\langle Tx,y\rangle|$ with $(x,y) \in B(X) \times B(X^*)$, while numerical radius $\|T\|_w=\sup| \langle Tx,\ y \rangle|$ with $(x,y) \in B(X) \times B(X^*)$ and $\langle x,y\rangle=1$, $\|T\|$ is bilinear and $\|T\|_w$ is quadratic in nature.  However, $\|T\|_w\leq \|T\|$ implies that there are more spaces for which $\|T\|\geq 1$ but $\|T\|_w=1$.  
\par
Furthermore, if $T$ is a bounded linear operator on a Hilbert space $H$, then the numerical radius takes the form
\begin{equation}
\|T\|_w=\sup\{|\langle Tx,x\rangle|:\|x\|=1\}.
\end{equation}
This follows from the fact that for each linear functional $x^*$ there is a unique $x_0 \in H$ such that $x^*(x)=\langle x,x_0\rangle$ for all $x \in H$.  Moreover, if $T$ is self-adjoint or a normal operator on a Hilbert space $H$, then 
\begin{equation}
\|T\|_w=\|T\|.
\end{equation}
Also, if a non-zero $T:H \to H$ is self-adjoint and compact, then $T$ has an eigenvalue $\lambda$ such that 
\begin{equation}
\|T\|_w=\|T\|=\lambda.
\end{equation}
These properties of numerical radius together with the desirable properties of diagonal projections from Hilbert spaces onto closed subspaces proved motivation to investigate minimal projections with respect to numerical radius.
\section{Characterization of Minimal Numerical Radius Projections}
In \cite{Aksoy-Chalmers}
a characterization of minimal numerical radius extension of operators from a Banach space $X$ onto its finite dimensional subspace $V=[v_1,v_2,\dots,v_n]$ is given.  To express this theorem, we first set up our notation.
\begin{notation}
Let $\displaystyle T=\sum_{i=1}^n u_i \otimes v_i: V \to V$ where $u_i \in V^*$ and its extension to $X$ is denoted by $\widetilde{T}:X \to V$ and defined as 
\begin{equation}
\widetilde{T}=\sum_{i=1}^n \tilde{u}_i \otimes v_i,
\end{equation}
where $\widetilde{u}_i \in X^*$.
\end{notation}
\begin{defi}
Let $X$ be a Banach space.  If $x \in X$ and $x^* \in X^*$ are such that 
\begin{equation}
|\langle x,x^*\rangle|=\|x\|\|x^*\|\neq 0,
\end{equation}
then $x^*$ is called an \emph{extremal} of $x$ and  written as  $x^*=ext \, x$.  Similarity, $x$ is an extremal of $x^*$.  We call $(ext\, y,y) \in S_{X^{**}}\times S_{X^*}$ a \emph{diagonal extremal pair} for $\widetilde{T} \in B(X,V)$ if 
\begin{equation}
\langle \widetilde{T}^{**}x,y\rangle=\|\widetilde{T}\|_w,
\end{equation}
where $\widetilde{T}^{**}:X^{**}\to V$ is the second adjoint extension of $\widetilde{T}$ are $V=[v_1,\dots,v_n]\subset X$.  In other words, the map $\widetilde{T}$ has the expression $\displaystyle \widetilde{T}=\sum_{i=1}^n \tilde{u}_i \otimes v_i: X \to V$ and 
\begin{equation}
\widetilde{T}x=\sum_{n=1}^n\langle x,\tilde{u}_i\rangle v_i
\end{equation}
where $\tilde{u}_i \in X^*$, $v_i \in V$ and $\langle x,\tilde{u}_i\rangle$ denotes the functional $\tilde{u}_i$ is acting on $x$ and 
\begin{equation}
\widetilde{T}^{**}x=\sum_{i=1}^n \langle u_i,x\rangle v_i,
\end{equation}
$u_i \in X^{***}$, $v_i \in V$, $x \in X^{**}$.
\end{defi}
The set of all diagonal extremal pairs will be denoted by $\mathcal{E}_w(\widetilde{T})$ and defined as:
\begin{equation}
\mathcal{E}_w(\widetilde{T})=\left\{(ext \, y,y) \in S_{X^{**}}\times S_{X^*} : \ \|\widetilde{T}\|_w=\sum_{i=1}^n\langle ext \, y,u_i\rangle \cdot \langle v_i,y\rangle\right\}.
\end{equation}
Note that to each $(x,y) \in X^{**} \times X^*$ we associate the rank-one operator $y \otimes x:X \to X^{**}$ given by
\begin{equation}
(y \otimes x)(z)=\langle z,y \rangle x {\hskip 2em} \text{for} \, \, z \in X.
\end{equation}
Accordingly, to each $(x,y) \in \mathcal{E}_w(\widetilde{T})$ we can associate the rank-one operator $y\otimes ext \, y:X \to X^{**}$ given by 
\begin{equation}
(y\otimes ext\, y)(z)=\langle z,y\rangle ext\, y.
\end{equation}
By $\mathcal{E}(\widetilde{T})$ we denote the usual set of all extremal pairs for $\widetilde{T}$ and 
\begin{equation}
\mathcal{E}(\widetilde{T})=\left\{(x,y)\in S_{X^{**}}\times S_{X^*}:\|\widetilde{T}\|=\sum_{i=1}^n\langle x,u_i\rangle \cdot \langle v_i,y\rangle \right\}.
\end{equation}
In case of diagonal extremal pairs we require $|\langle ext\, y,y\rangle |=1$. \begin{defi}
Let $\displaystyle T=\sum_{i=1}^n u_i \otimes v_i: V \to V=[v_1,v_2,\dots,v_n]\subset X$, where $u_i \in V^*$.  Let $\displaystyle\widetilde{T}:\sum{i=1}^n \tilde{u}_i\otimes v_i: X \to V$ be an extension of $T$ to all of $X$.  We say $\widetilde{T}$ is a \emph{minimal numerical extension} of $T$ if 
\begin{equation}
\|\widetilde{T}\|=\inf\left\{\|S\|_w:S:X \to V \ ; \ S_{|_V} =T\right\}.
\end{equation}
Clearly $\|T\|_w\leq \|\widetilde{T}\|_w.$
\end{defi}
\begin{theorem} (\cite{Aksoy-Chalmers})
$\widetilde{T}$ is a minimal radius-extension of $T$ if an only if the closed convex hull of $\{y\otimes x\}$ where $(x,y) \in \mathcal{E}_w(\widetilde{T})$ contains an operator for which $V$ is an invariant subspace.
\end{theorem}
\begin{theorem}
$P$ is a minimal projection from $X$ onto $V$ if and only if the closed convex hull of $\{y\otimes x\},$ where $(x,y) \in \mathcal{E}_w(P)$ contains an operator for which $V$ is an invariant subspace.
\end{theorem}
\begin{proof}
By taking $T=I$ and $\widetilde{T}=P$ one can appropriately modify the proof given in \cite{Aksoy-Chalmers} without much difficulty.  The problem is equivalent to the best approximation 
in the numerical radius of a fixed operator from the space of operator 
\begin{equation*}
\mathcal{D}=\{\Delta \in \mathcal{B}:\Delta=0 \ \text{on} \ V\}=sp\{\delta \otimes v: \delta \in V^\perp; v \in V\}.
\end{equation*}
  One of the main ingredients of the proof is Singer's identification theorem (\cite{Singer}, Theorem 1.1 (p.18) and Theorem 1.3 (p.29)) of finding the minimal operator 
as the error of best approximation in C(K) for $K$ Compact.  In the case 
of numerical radius, one considers $K_w=K \cap Diag=\{(x,y)\in B(X^{**})\times B(X^*):x=ext(y) \ \text{or} \ x=0\}$  and shows $K_w$ is compact.  Thus the set $\mathcal{E}(P)$, being the set of points where a continuous (bilinear) function achieves 
its maximum on a compact set, is not empty.  For further details see \cite{Aksoy-Chalmers}.
\end{proof}
\begin{theorem} (When minimal projections coincide)
In case $X=L^p$ for $p=1,2,\infty$, the minimal numerical radius projections and the minimal operator norm projections coincide with the same norms.
\end{theorem}
\begin{proof}
In case of $L^2$, for any self-adjoint operator, we have 
\begin{equation}
\|P\|=\|P\|_w=|\lambda|,
\end{equation}
where $\lambda$ is the maximum (in modulus) eigenvalue. In this case, 
\begin{equation}\|P\|=\|P\|_w=|\langle Px,x\rangle |,
\end{equation}
 where $x$ is a norm-1 ``maximum" eigenvector.  
\par
When $p=1,\infty$, it is well known that $n(L^p)=1$ (\cite{Bonsall-Duncan1}, section 9)
thus 
\begin{equation}
\|P\|=\|P\|_w.
\end{equation}
\end{proof}
\begin{exa}
The projection $P:l_3^p \to [v_1,v_2]=V$ where $v_1=(1,1,1)$ and $v_2=(-1,0,1)$ is minimal with respect to the operator 
norm, but not minimal with respect to numerical radius for $1<p<\infty$ and $p \neq 2$. Let us denote by $P_o,P_m$ projections minimal with respect to operator norm and numerical radius respectively.  In other words
\begin{align*}
\|P_o\|&=\inf\left\{\|P\|:P \in \mathcal{P}(X,V)\right\}\\
\|P_m\|_w&=\inf\left\{\|P\|_w:P \in \mathcal{P}(X,V)\right\}.
\end{align*}
Note that
\begin{equation}
P_o(f)=u_1(f)v_1+u_2(f)v_2 {\hskip 1em}\text{and}{\hskip 1em} P_m(f)=z_1(f)v_1+z_2(f)v_2.
\end{equation}
Then it is easy to see that
\begin{align*}
u_1=z_1=&\left(-\frac{1}{2},0,\frac{1}{2}\right)\\
u_2=&\left(\frac{1-d}{2},d,\frac{1-d}{2}\right)\\
z_2=&\left(\frac{1-g}{2},g,\frac{1-g}{2}\right),
\end{align*}
and for $p=\dfrac{4}{3}$ it is possible to determine $g$ and $d$ to conclude $\|P_o\|=1.05251$ while $\|P_m\|_w=1.02751$, thus $\|P_o\|\neq \|P_m\|_w$.
\end{exa}
V. P. Odinec in \cite{Odinec}  (see also \cite{OL}, \cite{LE})
proves that minimal projections of norm greater than one from a three-dimensional Banach space onto any of its two-dimensional subspaces are unique.  Thus in the above example, the projection from $l_3^p$ onto a two-dimensional subspace not only proves the fact that $\|P_o\|\neq \|P_m\|_w$ for $p\neq 1,2,\infty$, here once again we have the uniqueness of the projections.
\section{Rudin's Projection and Numerical Radius}
One of the key thoerems on minimal projections is due to W. Rudin (\cite{Rudin2} and \cite{Rudin1})
The setting for his theorem is as follows.  $X$ is a Banach space and $G$ is a compact topological group.  Defined on $X$ is a set $\mathcal{A}$ of all bounded linear bijective operators in a way that $\mathcal{A}$ is algebraically isomorphic to $G$.  The image of $g \in G$ under this isomorphism will be denoted by $T_g$.  We will assume that the map $G \times X \to X$ defined as $(g,x) \mapsto T_gx$ is continuous.  A subspace $V$ of $X$ is called \emph{G-invariant} if $T_g(V) \subset V$ for all $g \in G$ and a mapping $S: X \to X$ is said to \emph{commute} with $G$ if $S \circ T_g=T_g \circ S$ for all $g \in G$.  In case $\|T_g\|=1$ for all $g \in G$, we say $g$ acts on $G$ by \emph{isometries}.
\begin{theorem} (\cite{Rudin1})
Let $G$ be a compact topological group acting by isomorphism on a Banach space $X$ and let $V$ be a complemented $G$--invariant subspace of $X$.  If there exists a bound projection $P$ of $X$ onto $V$, then there exists a bounded linear projection $Q$ of $X$ onto $V$ which commutes with $G$.
\end{theorem}
The idea behind the proof of the above theorem is to obtain $Q$ by averaging the operators $T_{g^{-1}}PT_g$ with respect to Haar measure $\mu$ on $G$.  i.e.,
\begin{equation}
Q(x):=\int_G\left(T_{g^{-1}}PT_g\right)(x)\, d\mu(g).
\end{equation}
\par
Now assume $X$ has a norm which contains the maps $\mathcal{A}$ to be \emph{isometries} and all of the hypotheses in Rudin's theorem are satisfied,  then one can claim the following stronger version of Rudin's theorem :
\begin{cor}
If there is a \emph{unique} projection $Q: X \to V$ which commutes with $G$, then for any $P \in \mathcal{P}(X,V)$, the projection
\begin{equation}
Q(x)=\int_G\left(T_{g^{-1}}PT_g\right)(x)\, d\mu(g),
\end{equation}
is a \emph{minimal projection} of $X$ onto $V$.
\end{cor}
\begin{theorem}
\label{thm:star}
(\cite{Aksoy-Lewicki1})
Let $\mathcal{A}$ be a set of all bounded linear bijective operators on $X$ such that $\mathcal{A}$ is algebraically isomorphic to $G$.  Suppose that all of the hypotheses of Rudin's theorem above are satisfied and the maps in $\mathcal{A}$ are isometries.  If $P$ is any projection in the numerical radius of $X$ onto $V$, then the projection $Q$ defined as 
\begin{equation}
Q(x)=\int_G \left(T_{g^{-1}}PT_g\right)(x) \, d\mu(g)
\end{equation}
satisfies $\|Q\|_w\leq \|P\|_w.$ 
\end{theorem}
\begin{proof}
Consider $\|Q\|_w=\sup\{|x^*(Qx)|:x^*(Qx)\in W(Q)\}$, where $W(Q)$ is 
 the numerical range of $Q$.   Notice that
\begin{align}
|x^*(Qx)|=&\left|x^*\int_G\left(T_{g^{-1}}PT_g\right)(x)\, d\mu(g)\right|\notag\\
\leq& \int_G\left|\left(x^* \circ T_{g^{-1}}\right)P(T_gx)\right|\, d\mu(g).
\end{align}
But $\|x\|=1$ and $\|x^*\|=1$ which implies that $\|T_gx\|=1$ and $\|x^*T_{g^{-1}}\|=1$, moreover, 
\begin{equation}
1=x^*(x)=x^*T_{g^{-1}}(T_gx)\implies |x^*(Qx)|\leq \|P\|_w.
\end{equation}
Consequently, $\|Q\|_w \leq \|P\|_w$ which proves $Q$ is a minimal projection in numerical radius.
\end{proof}
\begin{theorem}
(\cite{Aksoy-Lewicki1})
Suppose all  hypotheses of the above theorem are satisfied and that there is exactly one projection $Q$ which commutes with $G$.  Then $Q$ is a minimal projection with respect to numerical radius.
\end{theorem}
\begin{proof}
Let $P \in \mathcal{P}(X,V)$.  By the properties of Haar measure, $Q_p$ given in the above theorem commutes with $G$.  Since there is exactly one projection which commutes with $G$, $Q_p=Q$ and $\|Q\|_w\leq \|P\|_w$ as desired.
\end{proof}

\begin{rem}
In \cite{Aksoy-Lewicki1} it is shown that if $G$ is a compact topological group acting by isometries on a Banach space $X$ and if we let 
\begin{equation}
\psi:B(X) \to [0,+\infty],
\end{equation}
be a convex function which is lower semi-continuous in the strong operator topology and if one further assumes that 
\begin{equation}
\psi\left(g^{-1}\circ P \circ g\right)\leq \psi(P),
\end{equation}
for some $P \in B(X)$ and $g \in G$, then $\psi(Q_P) \leq \psi(P).$ 
This result leads to calculation of minimal projections not only with respect to numerical radius but also with respect to $p$-summing, $p$-nuclear and $p$-integral norms.
For details see \cite{Aksoy-Lewicki1}.
\end{rem}
\section{An Application}
Let $C(2\pi)$ denote the set of all continuous $2\pi$-periodic functions and $\Pi_n$ be the space of all trigonometric polynomials of order $\leq n$ (for $n\geq 1)$.
\par
The \emph{Fourier projection} $F_n:C(2\pi)\to \Pi_n$ is defined by 
\begin{equation}
F_n(f)=\sum_{k=0}^{2n}\left(\int_0^{2\pi}f(t)g_n(t)dt\right)g_k,
\end{equation}
where $\displaystyle (g_k)_{k=0}^{2n}$ is an orthonormal basis in $\Pi_n$ with respect to the scalar product 
\begin{equation}
\langle f,g \rangle=\int_0^{2\pi}f(t)g(t)dt.
\end{equation}
Lozinskii in \cite{Lozinskii} showed that $F_n$ is a minimal projection in $\mathcal{P}(C(2\pi),\Pi_n)$.  His proof is based on the equality which states that for any $f \in C(2\pi)$, $t \in [0,2\pi]$ and $P \in \mathcal{P}(C(2\pi),\Pi_n)$, we have 
\begin{equation}
F_nf(t)=\frac{1}{2\pi}\int_0^{2\pi}\left(T_{g^{-1}}PT_gf\right)(t)\, d\mu(g).
\end{equation}
Here $\mu$ is a Lebesgue measure and $(T_gf)(t)=f(t+g)$ for any $g \in \R$.  This equality is called Marcinkiewicz equality (\cite{Cheney} p.233).
\par
Notice that $F_n$ is the only projection that commutes with $G$, where $G=[0,2\pi]$ with addition mod $2\pi$.  In particular, $F_n$ is a minimal projection with respect to numerical radius.
\par
Since we know the upper and lower bounds on the operator norm of $F_n$, more precisely (\cite{Cheney} p.212)
\begin{equation}
\frac{4}{\pi^2}ln(n)\leq \|F_n\|\leq ln(n)+3.
\end{equation}
From the theorem (when minimal projections coincide) we know that in cases of $L^p$, $p=1,\infty$, the numerical radius projections and the operator norm projections are equal.  Since $C(2\pi) \in L^\infty,$ we also have lower and upper bounds for the numerical radius of Fourier projections, i.e.,
\begin{equation}
\frac{4}{\pi^2}ln(n)\leq \|F_n\|_w\leq ln(n)+3.
\end{equation}
\begin{rem}
Lozinskii's proof of the minimality of $F_n$ is based on Marcinkiewicz equality.  However, the Marcinkiewicz equality holds true if one replaces $C(2\pi)$ by $L^p[0,2\pi]$ for $1\leq p \leq \infty$ or Orlicz space $L^\phi[0,2\pi]$ equipped with Luxemburg or Orlicz norm provided $\phi$ satisfies the suitable $\Delta_2$ condition.  Hence, Theorem \ref{thm:star} can be applied equally well to numerical radius or norm in Banach operator ideals of p-summing, p-integral,p-nuclear operators generated by $L^p$-norm or the Luxemburg or Orlicz norm.   For further examples see \cite{Aksoy-Lewicki1}.
\end{rem}
\section{Strongly Unique Minimal Extensions}
In  \cite{Odinec}  (see also \cite{OL}) it is shown that a minimal projection of the operator norm greater than one from a three dimensional real Banach space onto any of its two dimensional subspace is the unique minimal projection with respect to the operator norm. Later in \cite{LE} this result is generalized as follows: 
\\ Let $X$ is a three dimensional real Banach space and $V$ be its two dimensional subspace. Suppose  $A \in B(V)$ is a fixed operator. Set
$$
\mathcal{P}_{A}(X,V)= \{P\in B(X,V) : \,\,\, P\mid_{V}=A\,\,\}
$$
and  assume $\parallel P_0\parallel> \parallel A \parallel$, if $P_o \in \mathcal{P}_{A} (X,V)$ is an extension of minimal operator norm. Then $P_o$ is {\it a strongly unique minimal extension with respect to operator norm}.\\ In other words there exists $ r>0$ such that for all $P \in \mathcal {P} _{A} (X,V)$ one has
$$\| P \|\geq \| P_o \| +\,r \,\| P-P_o\| . $$
\begin{defi}
We say an operator $A_o \in \mathcal{P}_A(X,V)$ is {\it a strongly unique minimal extension with respect to numerical radius} if there exists $r > 0$ such that
$$
\| B \|_w \geq \| A_o \|_w +\,r \,\| B-A_o\|_w
$$
for any $B \in \mathcal{P}_{A}(X,V).$
\end{defi}
A natural extension of the above mentioned results to the case of numerical radius $\| \cdot \|_{w}$  was given in \cite{Aksoy-Lewicki2}.
\begin{theorem}
\label{application}
(\cite{Aksoy-Lewicki2})
Assume that $X$ is a three dimensional real Banach space and let $V$ be its two dimensional subspace. Fix $A \in B(V)$ with $\| A \|_w >0.$ Let
$$
\lambda_{w}^{A} = \lambda_{w}^{A}(V,X) = \inf \{ \| B \|_w : B \in \mathcal{P}_A(X,V) \} > \|A \|,
$$ 
where $\|A\|$ denotes the operator norm.
Then there exist exactly one $A_o \in \mathcal{P}_A(X,V)$ such that $$\lambda_{w}^{A} = \| A_o \|_w .$$ Moreover, $A_o$ is the strongly minimal extension with respect to numerical radius. 
\end{theorem}
Notice that if we take $ A = id_V$ then $ \|A\|_w = \|A\| =1.$ In this case Theorem (\ref{application}) reduces to the following theorem:
\begin{theorem}
\label{projections}
(\cite{Aksoy-Lewicki2})
Assume that $X$ is a three dimensional real Banach space and let $V$ be its two dimensional subspace. Assume that
$$
\lambda^{id_V}_w > 1.
$$
Then there exist exactly one $P_o \in \mathcal{P}(X,V)$ of minimal numerical radius. Moreover , $P_o$ is a strongly unique minimal projection with respect to numerical radius. In particular $P_o$ is the only one minimal projection with respect to the numerical radius.
\end{theorem}
\begin{rem}
\label{normone}
(\cite{Aksoy-Lewicki2})
Notice that in Theorem (\ref{application}) the assumption that $ \|A\| < \lambda_{w}^{A}$ is essential. Indeed, let $ X= l_{\infty}^{(3)},$ $ V =\{x \in X: x_1 +x_2=0\}$ and
$ A= id_V.$ Define
$$
P_1x = x -(x_1+x_2)(1,0,0)
$$
and
$$
P_2x = x -(x_1+x_2)(0,1,0).
$$
It is clear that
$$
\|P_1\| =\|P_1\|_w = \|P_2\| = \|P_2\|_w=1
$$
and $ P_1 \neq P_2.$ Hence, there is no strongly unique minimal projection with respect to numerical radius in this case.
\end{rem}
\begin{rem}
\label{dim4}
(\cite{Aksoy-Lewicki2})
Theorem (\ref{projections}) cannot be generalized for real spaces $X$ of dimension $ n \geq 4.$ Indeed let $ X= l_{\infty}^{(n)},$ and let $ V = ker(f),$
where $ f=(0,f_2,...,f_n) \in l_1^{(n)}$ satisfies $f_i>0$ for $i=2,...,n,$  $ \sum_{i=2}^n f_i=1$ and $ f_i <1/2 $ for $i=1,...,n.$ It is known
(see e.g.  \cite{bl}, \cite{OL}) that in this case
$$
\lambda(V,X)= 1 + (\sum_{i=2}^n f_i/(1-2f_i))^{-1} >1,
$$
where 
$$
\lambda(V,X)= \inf \{ \|P \|: P \in \mathcal(X,V)\}.
$$
By \cite{Aksoy-Chalmers}, $ \lambda(V,X) = \lambda_w^{id_V}.$ Define for $ i=1,...,n$ $ y_i= (\lambda(V,X)-1)(1-2f_i).$
Let $ y =(y_1,...,y_n)$ and $ z=(0,y_2,...,y_n).$ Consider mappings $P_1, P_2$ defined by
$$
P_1x = x-f(x)y
$$
and
$$
P_2x= x - f(x)z
$$
for $x \in l_{\infty}^{(n)}.$
It is easy to see that $ P_i \in \mathcal{P}(X,V),$ for $i=1,2,$ $ P_1 \neq P_2.$ By (\cite{OL} p. 104) $ \|P_i\| = \|P_i\|_w = \lambda(V,X) = \lambda_w^{id_V}.$
for $i=1,2.$
\end{rem}
\begin{rem}
\label{complex}
Theorem(\ref{projections}) is not valid for complex three dimensional spaces. For details see \cite{Aksoy-Lewicki2}.
\end{rem}
For Kolmogorov  type criteria concerning approximation with respect to numerical radius, we  refer the reader to  \cite{Aksoy-Lewicki2}.

{\vskip 1em}
Asuman G\"{u}ven Aksoy\\
Department of Mathematics\\
Claremont McKenna College\\
Claremont, CA, 91711, USA\\
e-mail: aaksoy@cmc.edu
{\vskip 1em}
\noindent Grzegorz Lewicki\\
Dept. of Math and Comp. Sci.\\
Jagiellonian University\\
Lojasiewicza 6\\
30-348 Krak\'ow, Poland\\
e-mail:Grzegorz.Lewicki@im.uj.edu.pl
\end{document}